\documentclass[12pt, reqno]{amsart}
\usepackage[T1]{fontenc}
\usepackage{amsfonts}
\usepackage{amsmath}
\usepackage{amsthm}
\usepackage{amssymb}
\usepackage{geometry}
\usepackage{graphicx}
\usepackage{xcolor}
\usepackage{mathtools}
\usepackage[colorlinks=true, linkcolor=blue]{hyperref}
\usepackage{cleveref}
\usepackage[english]{babel}
\usepackage{lineno}
\usepackage{float}

\newtheorem{theorem}{Theorem}[section]

\newtheorem{lemma}[theorem]{Lemma}
\newtheorem{conjecture}[theorem]{Conjecture}
\newtheorem{proposition}[theorem]{Proposition}
\newtheorem{problem}[theorem]{Problem}
\newtheorem{claim}[theorem]{Claim}

\usepackage{tikz}
\usetikzlibrary{positioning}
\usetikzlibrary{decorations,arrows}
\usetikzlibrary{decorations.markings}
\numberwithin{equation}{section}

\usepackage{rustic}
\usepackage[T1]{fontenc}

\emergencystretch=\maxdimen
\title{Strong edge-coloring of graphs with maximum edge weight seven}
\author{Runze Wang}
\address[]{Department of Mathematical Sciences, University of Memphis, Memphis, TN 38152, USA}
\email{rwang6@memphis.edu; runze.w@hotmail.com}
\thanks{}
\date{\today}

\begin{document}

\sloppy

\begin{abstract}
    A strong edge-coloring of a graph $G$ is an edge-coloring such that any two edges of distance at most two receive distinct colors. The minimum number of colors we need in order to give $G$ a strong edge-coloring is called the strong chromatic index of $G$, denoted by $\chi_s'(G)$. The maximum edge weight of $G$ is defined to be $\max\{d(u)+d(v):\ uv\in E(G)\}$. In this paper, using the discharging method, we prove that if $G$ is a graph with maximum edge weight $7$ and maximum average degree less than $\frac{40}{13}$, then $\chi_s'(G)\le 13$. Also, we determine the largest possible maximum average degree of a graph with given maximum edge weight.
\end{abstract}

\maketitle

\section{Introduction}
In this paper, every graph is assumed to be a finite simple connected graph. 

For a graph $G=(V,\ E)$, a \emph{proper edge-coloring} of $G$ is an assignment of colors to the edges such that adjacent edges receive distinct colors. On the basis of the proper edge-coloring, Fouquet and Jolivet \cite{FJ} introduced the \emph{strong edge-coloring} of $G$, where two adjacent edges receive distinct colors, and two edges joint by another edge also receive distinct colors. The distance between two edges $e_1$ and $e_2$ in $G$, denoted by $D_G (e_1,\ e_2)$, is defined to be the distance between the corresponding vertices in the line graph of $G$. So equivalently, we can say that, in a strong edge-coloring, any two edges of distance at most two receive distinct colors. 

If, by $f:\ E(G)\longrightarrow \{1,\ 2,\ ...,\ c\}$, we can use $c$ colors to give a strong edge-coloring to $G$, then $f$ is called a \emph{strong $c$-edge-coloring} of $G$. The minimum number of colors we need in a strong edge-coloring of $G$ is called the \emph{strong chromatic index} of $G$, denoted by $\chi_s'(G)$. 

Let $\Delta$ be the maximum vertex degree of $G$. The following conjecture was made by Erd\H os and Ne\v set\v ril \cite{EN}.

\begin{conjecture}[Erd\H os and Ne\v set\v ril \cite{EN}]\label{conj1} 
\[
    \chi_s'(G)\le\begin{cases}
        \frac{5}{4}\Delta^2 &if\ \Delta\ is\ even, \\
        \frac{5}{4}\Delta^2-\frac{\Delta}{2}+\frac{1}{4} &if\ \Delta\ is\ odd.
    \end{cases}
\]
\end{conjecture}

For general upper bounds: It was proved by Bruhn and Joos \cite{BJ} that $\chi_s'(G)\le 1.93\Delta^2$ for large $\Delta$. This result was improved to $\chi_s'(G)\le 1.772\Delta^2$ by Hurley et al. \cite{HDK}.

For small $\Delta$: This conjecture is trivially true when $\Delta\le 2$, and the first non-trivial case $\Delta=3$ was also verified \cite{An,HHT}. If $\Delta=4$, then the conjecture says that $\chi_s'(G)\le 20$, and the best result so far is $\chi_s'(G)\le 21$, given by Huang et al. \cite{HSY}. If $\Delta=5$, then the conjecture says that $\chi_s'(G)\le 29$, and the best result so far is $\chi_s'(G)\le 37$, given by Zang \cite{Za}.

The \emph{maximum average degree} of $G$, denoted by $mad(G)$, is defined by
\begin{align*}
    mad(G)=\max\bigl\{\bar{d}(G'):\ G'\subseteq G\bigr\},
\end{align*}
where $\bar{d}(G'):=\frac{2|E(G')|}{|V(G')|}$ is the average degree of $G'$.

It was proved by Lv et al. \cite{LLY} that if $G$ is a graph with $\Delta=4$ and $mad(G)<\frac{51}{13}$, then the conjecture is true for $G$. This gave an improvement on a similar result by Bensmail et al. \cite{BBH}, where they required $mad(G)<\frac{19}{5}$. It was proved by Lu et al. \cite{LLH} that if $G$ is a graph with $\Delta=5$ and $mad(G)<\frac{22}{5}$, then the conjecture is true for $G$. The proofs in \cite{LLY,BBH,LLH} are all based on the discharging method.

The \emph{maximum edge weight} of $G$, denoted by $W(G)$, is defined by
\begin{align*}
    W(G)=\max\{d(u)+d(v):\ uv\in E(G)\}.
\end{align*}
For graphs with fixed maximum edge weights, Chen et al. \cite{CHYZ} made the following conjecture.

\begin{conjecture}[Chen et al. \cite{CHYZ}]\label{conj2}
    Let $G$ be a graph with maximum edge weight $W(G)\ge 5$. Then
    \[\chi_s'(G)\le \begin{cases}
        5\lceil\frac{W(G)}{4}\rceil^2-8\lceil\frac{W(G)}{4}\rceil+3 &if\ W(G)\equiv 1\ mod\ 4; \\
        5\lceil\frac{W(G)}{4}\rceil^2-6\lceil\frac{W(G)}{4}\rceil+2 &if\ W(G)\equiv 2\ mod\ 4; \\
        5\lceil\frac{W(G)}{4}\rceil^2-4\lceil\frac{W(G)}{4}\rceil+1 &if\ W(G)\equiv 3\ mod\ 4; \\
        5\lceil\frac{W(G)}{4}\rceil^2 &if\ W(G)\equiv 0\ mod\ 4.
    \end{cases}
    \]
\end{conjecture}

This conjecture was verified for graphs with $W(G)=5$ by Wu and Lin \cite{WL}. It was also verified for graphs with $W(G)=6$ by Nakprasit and Nakprasit \cite{NN}, as well as independently by Chen et al. \cite{CHYZ}. If $W(G)=7$, then this conjecture says that $\chi_s'(G)\le 13$, and it was proved by Chen et al. \cite{CHYZ} that $\chi_s'(G)\le 15$. Very recently, Nelson and Yu \cite{NY} proved that this conjecture is true if $G$ is a planar graph with $W(G)=7$. If $W(G)=8$, then this conjecture says that $\chi_s'(G)\le 20$, and it was proved by Chen et al. \cite{CCZZ} that $\chi_s'(G)\le 21$.

In this paper, using the discharging method, we prove Conjecture \ref{conj2} for graphs with maximum edge weight seven and restricted maximum average degrees.

\begin{theorem}\label{thm}
    Let $G$ be a graph with $W(G)\le 7$. If $mad(G)<\frac{40}{13}$, then $\chi_s'(G)\le 13$.
\end{theorem}

Note that, in this theorem, we only require "$W(G)\le 7$". The conclusion for $G$ with "$W(G)=7$" naturally follows.

In the proof, we need to use a conclusion for bipartite graphs.

\begin{lemma}[Nakprasit \cite{Na}]\label{bipartite}
    Let $G$ be a bipartite graph with partition $(A,\ B)$. If every vertex in $A$ has degree at most $2$, and every vertex in $B$ has degree at most $d$, then $\chi_s'(G)\le 2d$.
\end{lemma}

The rest of this paper is organized as follows. In Section 2, we prove Theorem \ref{thm}. In Section 3, given the value of $W(G)$, we determine the largest possible value of $mad(G)$. In Section 4, we suggest two future goals.

\section{The proof of Theorem \ref{thm}}
It is easy to see that, if $G'$ is a subgraph of $G$, then $mad(G')\le mad(G)$ and $W(G')\le W(G)$.

Suppose that Theorem \ref{thm} is not true, and $H$ is a minimal counterexample to Theorem \ref{thm}, that is, $H$ is a graph with $W(H)\le 7$, $mad(H)<\frac{40}{13}$, and $\chi_s'(H)\ge 14$, but if we delete a vertex (and the edges on it) from $H$ and get a subgraph $H'$, then we will have $\chi_s'(H')\le 13$. 

For a strong $13$-edge-coloring $f:\ E(H')\longrightarrow \{1,\ 2,\ ...,\ 13\}$ of $H'$, we call $f$ a \emph{partial coloring} of $H$. For an edge $e\in E(H)\setminus E(H')$, let $L_f (e)$ be the set of the colors that can be used on $e$, i.e. the colors that have not been used on an edge in $H'$ which is of distance at most two from $e$. Observe that, if there are only $m\le 12$ edges in $H'$ of distance at most two from $e$, then $|L_f(e)|\ge 13-m$.

\begin{lemma}\label{lemma}
    Let $v$ be a vertex in $H$ with $k$ edges $vv_1,\ vv_2,\ ...,\ vv_k$ on it. We delete $v$ and these $k$ edges from $H$ to get a subgraph $H'$, so we have $E(H)\setminus E(H')=\{vv_1,\ vv_2,\ ...,\ vv_k\}$. Let $f:\ E(H')\longrightarrow \{1,\ 2,\ ...,\ 13\}$ be a partial coloring of $H$. If, up to re-ordering $vv_1,\ vv_2,\ ...,\ vv_k$, we have 
    \begin{align*}
        |L_f (vv_i)|\ge i
    \end{align*}
    for each $1\le i\le k$, then we can extend partial coloring $f$ to a strong $13$-edge-coloring of $H$.
\end{lemma}

\begin{proof}
    As $H'$ is a subgraph constructed by deleting a vertex $v$ and all the edges on $v$ from $H$, we have that for $e_1,\ e_2\in E(H')$, if $D_{H'}(e_1,\ e_2)\ge 3$, then $D_H (e_1,\ e_2)\ge 3$. This is because 
    \begin{itemize}
        \item if $D_H (e_1,\ e_2)=1$, then we must also have $D_{H'}(e_1,\ e_2)=1$, a contradiction;
        \item if $D_H (e_1,\ e_2)=2$ but $D_{H'}(e_1,\ e_2)\ge 3$, then there is an edge $vv_i\in E(H)\setminus E(H')$ such that $v\in e_1$ and $v_i\in e_2$, but $v\in e_1$ implies $e_1\in E(H)\setminus E(H')$, which contradicts the assumption that $e_1\in E(H')$.
    \end{itemize}
    
    So we can let each edge in $H'$ keep its color under $f$. Then we color $vv_1$ by a color in $L_f (vv_1)$. For $2\le j\le k$, after coloring $vv_{j-1}$, there is still at least $|L_f (vv_j)|-(j-1)\ge j-(j-1)=1$ color that can be used on $vv_j$. So we can color $vv_2,\ vv_3,\ ...,\ vv_k$ in order and get a strong $13$-edge-coloring of $H$.
\end{proof}

Now we use Lemma \ref{lemma} to prove some structural results about $H$. If a vertex has degree $d$, then we will call this vertex a $d$-vertex.

\begin{claim}\label{1v}
    There are no $1$-vertices in $H$.
\end{claim}

\begin{proof}
    Assume that $v$ is a $1$-vertex and $e$ is the only edge on $v$. Deleting $v$ and $e$ from $H$, we get a subgraph $H'$ with a strong $13$-edge-coloring $f$, which is a partial coloring of $H$. It is easy to check that, in $E(H')$, there are at most $9$ edges of distance at most $2$ from $e$, so $L_f(e)\ge 4$, hence by Lemma \ref{lemma}, we can extend $f$ to a strong $13$-edge-coloring of $H$, a contradiction.
\end{proof}

\begin{claim}
    There are no $6$-vertices in $H$.
\end{claim}
\begin{proof}
    This is clear because, under the assumption that $W(H)\le 7$, a $6$-vertex can only be adjacent to $1$-vertices.
\end{proof}

\begin{claim}\label{2v}
    A $2$-vertex in $H$ is adjacent to two $4$-vertices.
\end{claim}
\begin{proof}
    Assume that $v$ is a $2$-vertex and $e_1=vv_1$ and $e_2=vv_2$ are the two edges on $v$. Deleting $v$, $e_1$, and $e_2$ from $H$, we get a subgraph $H'$ with a strong $13$-edge-coloring $f$. By symmetry, we may assume that $d(v_1)\le d(v_2)$. 
    
    First, for $(d(v_1),\ d(v_2))\in\{(2,\ 2),\ (2,\ 3),\ (2,\ 4),\ (2,\ 5),\ (3,\ 3),\ (3,\ 4),\ (3,\ 5)\}$, by checking the number of edges in $H'$ of distance at most $2$ from $e_1$ and the number of edges in $H'$ of distance at most $2$ from $e_2$, we have

    \begin{table}[H]
\begin{tabular}{|l|l|l|l|l|l|l|l|}
\hline
$(d(v_1),\ d(v_2))$           & $(2,\ 2)$ & $(2,\ 3)$ & $(2,\ 4)$ & $(2,\ 5)$ & $(3,\ 3)$ & $(3,\ 4)$ & $(3,\ 5)$ \\ \hline
$(L_f(e_1)\ge,\ L_f(e_2)\ge)$ & $(7,\ 7)$ & $(6,\ 4)$ & $(5,\ 3)$ & $(4,\ 4)$ & $(3,\ 3)$ & $(2,\ 2)$ & $(1,\ 3)$ \\ \hline
\end{tabular}
\end{table}
    So by Lemma \ref{lemma}, in each case, we can extend $f$ to a strong $13$-edge-coloring of $H$, a contradiction.

    Then, assume that $(d(v_1),\ d(v_2))=(4,\ 5)$. In this case, in $H'$, there are at most $13$ edges of distance at most $2$ from $e_1$, and at most $11$ edges of distance at most $2$ from $e_2$. Let $e_3\neq e_2$ be another edge on $v_2$. In $H'$, excluding $e_3$ itself, there are at most $11$ edges of distance at most $2$ from $e_3$. Now we erase the color of $e_3$ under $f$ and get a new partial coloring $f'$. Under $f'$, we have $L_{f'}(e_1)\ge 1$, $L_{f'}(e_2)\ge 3$, and $L_{f'}(e_3)\ge 2$, so by Lemma \ref{lemma}, we can find feasible colors for these three edges in the order of $e_1$, $e_3$, $e_2$, and extend $f'$ to a strong $13$-edge-coloring of $H$, a contradiction. 
    
    Finally, assume that $(d(v_1),\ d(v_2))=(5,\ 5)$. As $W(H)\le 7$, a $5$-vertex in $H$ can only be adjacent to $1$-vertices and $2$-vertices. By Claim \ref{1v}, we know there are no $1$-vertices in $H$. So every $5$-vertex in $H$ is adjacent to five $2$-vertices. As there are no $2$-vertices whose one neighbor is of degree at most $4$ and the other neighbor is a $5$-vertex, we know that $H$ is a bipartite graph with partition $(A,\ B)$, where all vertices in $A$ are $2$-vertices and all vertices in $B$ are $5$-vertices. Then by Lemma \ref{bipartite}, we have $\chi_s'(H)\le 10$, a contradiction.

    So $(d(v_1),\ d(v_2))=(4,\ 4)$ is the only possible case.
\end{proof}

We would like to mention that the "erasing a color" technique is also employed in \cite{LLH} and \cite{LLY}.

In $H$, $5$-vertices can only be adjacent to $1$-vertices and $2$-vertices. But $1$-vertices do not exist by Claim \ref{1v}, and $2$-vertices are only adjacent to $4$-vertices by Claim \ref{2v}, so there are no $5$-vertices.

\begin{claim}\label{5v}
    There are no $5$-vertices in $H$.
\end{claim}

So, there are only $2$-vertices, $3$-vertices, and $4$-vertices in $H$.

Then, by Claim \ref{2v}, we know that the $3$-vertices in $H$ can only be adjacent to $3$-vertices and $4$-vertices. We divide the $3$-vertices into four classes $A$, $B$, $C$, and $D$, and a $3$-vertex will be called
\begin{itemize}
    \item a $3(A)$-vertex if it is adjacent to three $4$-vertices;
    \item a $3(B)$-vertex if it is adjacent to two $4$-vertices and one $3$-vertex;
    \item a $3(C)$-vertex if it is adjacent to one $4$-vertex and two $3$-vertices;
    \item a $3(D)$-vertex if it is adjacent to three $3$-vertices.
\end{itemize}

\begin{claim}\label{3v1}
    A $3(D)$-vertex in $H$ is adjacent to at most one $3(D)$-vertex.
\end{claim}

\begin{proof}
    Assume on the contrary that $u$ is a $3(D)$-vertex and it is adjacent to two $3(D)$-vertices $v$ and $w$, and another $3$-vertex $x$. Deleting $u$ and the three edges $uv$, $uw$, and $ux$, we get a subgraph $H'$ with a strong $13$-edge-coloring $f$. In $H'$, there are at most $10$ edges of distance at most $2$ from $uv$, at most $10$ edges of distance at most $2$ from $uw$, and at most $12$ edges of distance at most $2$ from $ux$. So $L_f(uv)\ge 3$, $L_f(uw)\ge 3$, and $L_f(ux)\ge 1$. By Lemma \ref{lemma}, we can extend $f$ to a strong $13$-edge-coloring of $H$, a contradiction.
\end{proof}

\begin{claim}\label{3v2}
    If a $3(D)$-vertex is adjacent to another $3(D)$-vertex, then it is adjacent to two $3(B)$-vertices. 
\end{claim}

\begin{proof}
    Assume that $u$ is a $3(D)$-vertex and it is adjacent to another $3(D)$-vertex $v$ and two other $3$-vertices $w$ and $x$. Assume on the contrary that $w$ is a $3(C)$-vertex. Deleting $u$ and the three edges $uv$, $uw$, and $ux$, we get a subgraph $H'$ with a strong $13$-edge-coloring $f$. In $H'$, there are at most $10$ edges of distance at most $2$ from $uv$, at most $11$ edges of distance at most $2$ from $uw$, and at most $12$ edges of distance at most $2$ from $ux$. So $L_f(uv)\ge 3$, $L_f(uw)\ge 2$, and $L_f(ux)\ge 1$. By Lemma \ref{lemma}, we can extend $f$ to a strong $13$-edge-coloring of $H$, a contradiction.
\end{proof}

Then we look at the $4$-vertices.

\begin{claim}\label{4v}
    A $4$-vertex in $H$ is adjacent to at most one $2$-vertex.
\end{claim}

\begin{proof}
    Assume on the contrary that $u$ is a $4$-vertex and it is adjacent to two $2$-vertices $v$ and $w$, and two other vertices $x$ and $y$. Deleting $u$ and the four edges $uv$, $uw$, $ux$, and $uy$, we get a subgraph $H'$ with a strong $13$-edge-coloring $f$. 
    
    Recall that, by Claim \ref{5v}, we do not have $5$-vertices in $H$. So in $H'$, there are at most $9$ edges of distance at most $2$ from $uv$, at most $9$ edges of distance at most $2$ from $uw$, at most $12$ edges of distance at most $2$ from $ux$, and at most $12$ edges of distance at most $2$ from $uy$. Let $z\neq u$ be the other vertex adjacent to $v$ (it is possible that $z\in\{w,\ x,\ y\}$). In $H'$, excluding $vz$ itself, there are at most $9$ edges of distance at most $2$ from $vz$. Now we erase the color of $vz$ under $f$ and get a new partial coloring $f'$. Under $f'$, we have $L_{f'}(uv)\ge 5$, $L_{f'}(uw)\ge 5$, $L_{f'}(ux)\ge 2$, $L_{f'}(uy)\ge 2$, and $L_{f'}(vz)\ge 4$, so by Lemma \ref{lemma}, we can find feasible colors for these five edges, and extend $f'$ to a strong $13$-edge-coloring of $H$, a contradiction.
\end{proof}

Now, for $d\in\{2,\ 3,\ 4\}$, we let each $d$-vertex $v$ in $H$ have initial charge $\omega(v)=d-\frac{40}{13}$. We have assumed that $mad(H)<\frac{40}{13}$, so $\sum_{v\in V(H)}\omega(v)<0$. We make the following discharging rules.
\begin{itemize}
    \item Each $4$-vertex gives $\frac{7}{13}$ to each $2$-vertex adjacent to it, and gives $\frac{5}{39}$ to each $3$-vertex adjacent to it.
    \item Each $3(B)$-vertex gives $\frac{1}{26}$ to the $3$-vertex adjacent to it.
    \item Each $3(C)$-vertex gives $\frac{1}{39}$ to each $3$-vertex adjacent to it.
\end{itemize}

After the discharging process, each vertex $v$ gets a new charge $\omega'(v)$, and the total charge in $H$ does not change.

By all the claims we made, we can check that $\omega'(v)\ge 0$ for every $v\in V(H)$ as follows. Recall that a $4$-vertex has at most one $2$-neighbor by Claim \ref{4v}.
\begin{itemize}
    \item If $v$ is a $4$-vertex which is adjacent to one $2$-vertex and three $3$-vertices, then $\omega'(v)=4-\frac{40}{13}-\frac{7}{13}-3\cdot\frac{5}{39}=0$.
    \item If $v$ is a $4$-vertex which is adjacent to four $3$-vertices, then $\omega'(v)=4-\frac{40}{13}-4\cdot\frac{5}{39}>0$.
    \item If $v$ is a $3(A)$-vertex, then $\omega'(v)=3-\frac{40}{13}+3\cdot\frac{5}{39}>0$.
    \item If $v$ is a $3(B)$-vertex, then $\omega'(v)\ge 3-\frac{40}{13}+2\cdot\frac{5}{39}-\frac{1}{26}>0$.
    \item If $v$ is a $3(C)$-vertex, then $\omega'(v)\ge 3-\frac{40}{13}+\frac{5}{39}-2\cdot\frac{1}{39}=0$.
    \item If $v$ is a $3(D)$-vertex which is adjacent to another $3(D)$-vertex, then by Claim \ref{3v2}, we have $\omega'(v)=3-\frac{40}{13}+2\cdot\frac{1}{26}=0$.
    \item If $v$ is a $3(D)$-vertex which is not adjacent to a $3(D)$-vertex, then $\omega'(v)\ge 3-\frac{40}{13}+3\cdot\frac{1}{39}=0$.
    \item If $v$ is a $2$-vertex, then by Claim \ref{2v}, we have $\omega'(v)=2-\frac{40}{13}+2\cdot\frac{7}{13}=0$.
\end{itemize}

Now, for every vertex $v$ in $H$, we have $\omega'(v)\ge 0$, so $0\le \sum_{v\in V(H)}\omega'(v)=\sum_{v\in V(H)}\omega(v)<0$, a contradiction. Thus, such a minimal counterexample $H$ does not exist, and Theorem \ref{thm} is true.

\section{On the value of $mad(G)$}
In Section 1, we mentioned some results for graphs with restricted maximum average degrees. For example, it was proved by Lu et al. \cite{LLH} that if $G$ is a graph with $\Delta(G)=5$ and $mad(G)<\frac{22}{5}$, then Conjecture \ref{conj1} is true. Clearly, if $\Delta(G)=5$, then the largest possible value of $mad(G)$ is $5$, so we can say that the bound $mad(G)<\frac{22}{5}$ is $\frac{3}{5}$ away from the best possible bound $mad(G)\le 5$.

So, how good is our bound $mad(G)<\frac{40}{13}$? What is the largest possible value of $mad(G)$ if $G$ is a graph with $W(G)=7$?

\begin{proposition}
    If $W(G)=2k+1$ is odd, then
    \begin{align*}
        mad(G)\le \frac{2k(k+1)}{2k+1}.
    \end{align*}
    If $W(G)=2k$ is even, then 
    \begin{align*}
        mad(G)\le k.
    \end{align*}
    Both equalities can be attained.
\end{proposition}

The results in this proposition are intuitive. Here we give a formal proof of the odd case, and it is very similar for the even case.

If $W(G)=2k+1$, then the degree of a vertex in $G$ is between $1$ and $2k$. Note that here we must have $k\ge 1$, because it is not possible to have $W(G)=1$. Let $G'$ be a subgraph of $G$. For $1\le \ell\le 2k$, let $V_\ell$ be the set of $\ell$-vertices in $G'$.  We have $W(G')\le 2k+1$, so for any $j\le k$, a $(2k+1-j)$-vertex can only be adjacent to $i$-vertices with $i\le j$. This means for any $1\le t\le k$, we have
\begin{align*}
    \sum_{i=1}^t i|V_i|\ge \sum_{j=1}^t (2k+1-j)|V_{2k+1-j}|.
\end{align*}

For the average degree of $G'$, we have
\begin{align*}
    \bar{d}(G')=\frac{\sum_{\ell=1}^{2k} \ell|V_\ell|}{\sum_{l=1}^{2k}|V_l|},
\end{align*}

Our goal is to show $\bar{d}(G')\le\frac{2k(k+1)}{2k+1}$ for any subgraph $G'$ of $G$. It suffices to prove the following more general result.

\begin{proposition}
    Let $X_1,\ X_2,\ ...,\ X_{2k}$ be non-negative integers with $\sum_{l=1}^{2k}X_l>0$. Assume that we have
    \begin{align}
        1 X_1&\ge 2k X_{2k}, \label{ineq1}\\
        1 X_1+2 X_2&\ge 2k X_{2k}+(2k-1) X_{2k-1}, \label{ineq2}\\
        &... \nonumber \\
        1 X_1+2 X_2+...+(k-1) X_{k-1}&\ge 2k X_{2k}+(2k-1) X_{2k-1}+...+(k+2) X_{k+2}, \label{ineqk-1}\\
        1 X_1+2 X_2+...+k X_k&\ge 2k X_{2k}+(2k-1) X_{2k-1}+...+(k+1)X_{k+1}.\label{ineqk}
    \end{align}
    Then 
    \begin{align*}
        \frac{\sum_{\ell=1}^{2k} \ell X_\ell}{\sum_{l=1}^{2k}X_l}\le \frac{2k(k+1)}{2k+1}.
    \end{align*}
\end{proposition}

\begin{proof}
    Let $\sigma_1=\sum_{i=1}^k X_i$ and $\sigma_2=\sum_{j=1}^k X_{2k+1-j}$.
    
    First, by inequality \eqref{ineqk}, it is clear that
    \begin{align*}
        k\sigma_1&\ge (k+1)\sigma_2.
    \end{align*}
    Here we have $\sigma_1>0$ because $\sigma_1+\sigma_2=\sum_{l=1}^{2k}X_l>0$. So
    \begin{align}
        \frac{\sigma_2}{\sigma_1}&\le \frac{k}{k+1}. \label{ineqa}
    \end{align}
    Then, we will show that 
    \begin{align}
        k\sigma_1+(k+1)\sigma_2\ge \sum_{\ell=1}^{2k} \ell X_\ell, \label{ineqb}
    \end{align}
    which is equivalent to
    \begin{align}
        (k-1)X_1+(k-2)X_2+...+1X_{k-1}\ge(k-1)X_{2k}+(k-2)X_{2k-1}+...+1X_{k+2}.\label{ineqc}
    \end{align}

    If $k=1$, then what we need is $0\ge 0$, which is trivially true. Then, to show inequality \eqref{ineqc} for $k\ge 2$, we take the sum of the $k-1$ inequalities \eqref{ineq1}, \eqref{ineq2}, ..., \eqref{ineqk-1}, and get
    \begin{align*}
        &(k-1)1X_1+(k-2)2X_2+...+1(k-1)X_{k-1}\\
        \ge& (k-1)2kX_{2k}+(k-2)(2k-1)X_{2k-1}+...+1(k+2)X_{k+2}.
    \end{align*}
    Dividing it by $k-1$, we get
    \begin{align}
        &(k-1)\frac{1}{k-1}X_1+(k-2)\frac{2}{k-1}X_2+...+1\frac{k-1}{k-1}X_{k-1} \nonumber\\
        \ge& (k-1)\frac{2k}{k-1}X_{2k}+(k-2)\frac{2k-1}{k-1}X_{2k-1}+...+1\frac{k+2}{k-1}X_{k+2}.\label{ineqd}
    \end{align}
    Observe that
    \begin{align*}
        (k-1)X_1+(k-2)X_2+...+1X_{k-1}\ge LHS\ in\ \eqref{ineqd},
    \end{align*}
    and 
    \begin{align*}
        RHS\ in\ \eqref{ineqd}\ge (k-1)X_{2k}+(k-2)X_{2k-1}+...+1X_{k+2}.
    \end{align*}
    So inequality \eqref{ineqc} is true, hence inequality \eqref{ineqb} is true.
    
    Inequality \eqref{ineqb} tells us that 
    \begin{align*}
        \frac{\sum_{\ell=1}^{2k} \ell X_\ell}{\sum_{l=1}^{2k}X_l}&\le \frac{k\sigma_1+(k+1)\sigma_2}{\sum_{l=1}^{2k}X_l}\\
        &=\frac{k\sigma_1+(k+1)\sigma_2}{\sigma_1+\sigma_2}\\
        &=\frac{k+(k+1)\frac{\sigma_2}{\sigma_1}}{1+\frac{\sigma_2}{\sigma_1}}\\
        &=k+1-\frac{1}{1+\frac{\sigma_2}{\sigma_1}}.
    \end{align*}
    Then by inequality \eqref{ineqa}, we have
    \begin{align*}
        \frac{\sum_{\ell=1}^{2k} \ell X_\ell}{\sum_{l=1}^{2k}X_l}\le \frac{2k(k+1)}{2k+1}.
    \end{align*}
\end{proof}

We can take $G$ to be the complete bipartite graph $K_{k+1,\, k}$, and then we have $mad(G)=\frac{2k(k+1)}{2k+1}$, so the equality can be attained. In particular, if $W(G)=7$, then $k=3$ and the largest possible value of $mad(G)$ is $\frac{24}{7}$. So, our bound $mad(G)<\frac{40}{13}$ is $\frac{32}{91}$ away from the best possible bound. 

\section{Conclusion}

In this paper, with a discharging method, we have shown that if $G$ is a graph with $W(G)\le 7$ and $mad(G)<\frac{40}{13}$, then $\chi_s'(G)\le 13$. The first future goal is to improve the bound on $mad(G)$.
\begin{problem}
    Improve the bound $mad(G)<\frac{40}{13}$.
\end{problem}

Also, it is natural to consider the next case $W(G)=8$, following the established pattern.
\begin{problem}
    Find an $M$ such that if $G$ is a graph with $W(G)\le 8$ and $mad(G)<M$, then $\chi_s'(G)\le 20$.
\end{problem}


\section*{Data availability}
There are no data associated with this paper.

\section*{Declarations}
\textbf{Competing Interests:} The author declares that he has no competing interests.

\end{document}